\title{\vspace{-0.6cm} On local Tur\'an problems}

\documentclass[12pt]{article}

\usepackage{amsmath, amssymb, amsthm, url, tikz, verbatim, longtable, xcolor}

\usepackage{blkarray}

\newtheorem{theorem}{Theorem}[section]
\newtheorem{proposition}[theorem]{Proposition}
\newtheorem{lemma}[theorem]{Lemma}

\newtheorem{conjecture}[theorem]{Conjecture}
\newtheorem{definition}[theorem]{Definition}

\usetikzlibrary{calc}
\usetikzlibrary{patterns}
\usetikzlibrary{decorations.markings}
\usetikzlibrary{arrows,shapes.geometric,%
	decorations.pathreplacing,shapes,shadows}

\date{}

\author{Peter Frankl \thanks{Renyi Institute, Budapest. Email: peter.frankl@gmail.com.} \and Hao Huang \thanks{Department of Mathematics, Emory University, Atlanta, USA. Email: hao.huang@emory.edu. Research supported in part by a Collaboration Grant from the Simons Foundation, NSF CAREER grant DMS-1945200, and an Alfred P. Sloan Fellowship.} \and Vojt\v ech R\"odl \thanks{Department of Mathematics, Emory, University, Atlanta, USA. Email: vrodl@emory.edu. Research supported in part by NSF grant DMS-1764385.}}

\begin{document}

\maketitle
\abstract{Since its formulation, Tur\'an's hypergraph problems have been among the most challenging open problems in extremal combinatorics. One of them is the following: given a $3$-uniform hypergraph $\mathcal{F}$ on $n$ vertices in which any five vertices span at least one edge, prove that $|\mathcal{F}| \ge (1/4 -o(1))\binom{n}{3}$. The construction showing that this bound would be best possible is simply $\binom{X}{3} \cup \binom{Y}{3}$ where $X$ and $Y$ evenly partition the vertex set. This construction has the following more general $(2p+1, p+1)$-property: any set of $2p+1$ vertices spans a complete sub-hypergraph on $p+1$ vertices. One of our main results says that, quite surprisingly, for all $p>2$ the $(2p+1,p+1)$-property implies the conjectured lower bound.

\section{Introduction}
Let $X$ be a finite set and $\binom{X}{r}$ the collection of all its $r$-subsets. Subsets $\mathcal{H}$ of $\binom{X}{r}$ are called $r$-uniform hypergraphs. Members of $\mathcal{H}$ are called edges. If $\binom{Y}{r} \subset \mathcal{H}$, then $Y$ is said to be a clique and $|Y|$ is its size. We denote by $K_t^r$ the $r$-uniform $t$-vertex clique. Note that every edge is a clique of size $r$. 

For integers $q \ge p \ge r \ge 2$, we say that $\mathcal{H}$ has property $(q, p)$ if for every $Z \in \binom{X}{q}$ there exists $Y \subset \binom{Z}{p}$ spanning a clique in $\mathcal{H}$, that is, $\binom{Y}{r} \subset \mathcal{H}$.

\begin{definition}
Let $T_r(n, q, p)=\min\{|\mathcal{H}|: \mathcal{H} \subset \binom{[n]}{r},~\mathcal{H}~\textrm{has~property~}(q, p)\}$. Set also $t_r(n, q, p)=T_r(n, q, p)/\binom{n}{r}$.
\end{definition}

Eighty years ago, Tur\'an \cite{turan} determined $T_2(n, q, 2)$ and this result served as the starting point for a lot of research that led to the creation of the field of extremal graph theory. About two decades later Tur\'an \cite{turan_conj} proposed two conjectures concerning $T_3(n, 4, 3)$ and $T_3(n, 5, 3)$. To state their asymptotic forms, let us mention that Katona, Nemetz and Simonovits \cite{KNS} used a simple averaging argument to show that $t_r(n, q, p)$ is monotone increasing as a function of $n$. Consequently the limit
$$\lim_{n \rightarrow \infty} t_r(n, q, p)=:t_r(q, p)$$ exists. 
\begin{conjecture}(Tur\'an)
\begin{equation}\label{eq_k43}
t_3(4, 3)=\frac{4}{9}.
\end{equation} 

\begin{equation}\label{eq_k53}
t_3(5, 3)=\frac{1}{4}.
\end{equation} 

\end{conjecture}
Even though this conjecture has been around for quite a long time, neither statement was proved. For \eqref{eq_k43} the best known bound stands as $t_3(4, 3) \ge 0.438334$ by Razborov \cite{razborov} using flag algebra. As for \eqref{eq_k53}, the construction providing the upper bound is very simple, namely $\mathcal{H}=\binom{X_1}{3} \cup \binom{X_2}{3}$, with $X_1 \sqcup X_2=[n]$, $|X_1|=\lceil \frac{n}{2} \rceil$, $|X_2|=\lfloor \frac{n}{2} \rfloor$.

Let us mention that in \cite{erdos-spencer} it was shown that for the graph case, 
\begin{equation}\label{eq_graph}
t_2(q, p)=1/\left\lfloor \frac{q-1}{p-1} \right\rfloor.
\end{equation}
For general $r$, Frankl and Stechkin \cite{frankl-stechkin} proved that
\begin{equation}\label{eq_small_q}
t_r(q, p)=1~~~~\textrm{if}~q \le \frac{r}{r-1} (p-1).
\end{equation}
It is easy to check that $\mathcal{H}=\binom{X_1}{r} \cup \binom{X_2}{r}$ has property $(2p+1, p+1)$ for all $p \ge r-1$. Consequently, 
\begin{equation}\label{eq_upper}
t_r(2p+1, p+1) \le \frac{1}{2^{r-1}}.
\end{equation}
For the case $r=3$, it was proved by the first author \cite{frankl} that
\begin{equation}\label{eqn_lim_3}
\lim_{p \rightarrow \infty} t_3(2p+1, p+1)=\frac{1}{4}.
\end{equation}

By developing the methods used in \cite{frankl}, in Section \ref{sec_main} we generalize \eqref{eqn_lim_3} to the $r$-uniform case.
\begin{theorem}\label{thm_limit}
For integers $r \ge 2 $ and $a \ge 2$, 
$$\lim_{p \rightarrow \infty} t_r(ap+1, p+1)=\frac{1}{a^{r-1}}.$$ 
\end{theorem}
In the $3$-uniform case (when $r=3$),  we are able to determine the exact value of $t_3(2p+1, p+1)$, for all $p \ge 3$, which strengthens \eqref{eqn_lim_3}.
\begin{theorem}\label{thm_3_uniform}
	
For every integer $p \ge 3$, 
$$t_3(2p+1, p+1)=\frac{1}{4}.$$ 
\end{theorem}
We should remark that the proof of this result is relying on earlier Tur\'an-type results of Mubayi and R\"odl \cite{MR}, and Baber and Talbot \cite{BT}. We are going to state these results in Section \ref{sec_3} before proving Theorem \ref{thm_3_uniform}. In Section \ref{sec_cr} we mention some open problems.

\section{Proof of Theorem \ref{thm_limit}}\label{sec_main}
Throughout the proof of Theorem \ref{thm_limit}, we assume $r \ge 3$, and $a \ge 2$ to be fixed, since the $r=2$ case is already covered by \eqref{eq_graph}. With $r$ fixed, we also set $t(q, p)=t_r(q, p)$. For the pair $(q, p)$ with $q \le ap$, we call $ap-q$ the {\it excess} $e(q, p)$ of the pair $(q, p)$. Note that since $q \ge p$, we always have $e(q, p) \le aq-q=(a-1)q$. For $\mathcal{F} \subset \binom{Y}{r}$, the set $Z$ is a $(w, v)$-{\it hole} if $|Z|=w$, the clique number of $\mathcal{F}|_Z$ (the sub-hypergraph of $\mathcal{F}$ induced by $Z$) is $v$, and $w>av$. We first establish the following two lemmas. 

\begin{lemma} \label{lem_hole}
Suppose $\mathcal{G} \subset \binom{Y}{r}$ has property $(q, p)$, and $Z$ is a $(w, v)$-hole of $\mathcal{G}$ with $w<q$, then $\mathcal{G}|_{Y\setminus Z}$ has property $(q-w, p-v)$.
\end{lemma}
\begin{proof}
Take an arbitrary set $U \in \binom{Y\setminus Z}{q-w}$, then $U \cup Z \in \binom{Y}{q}$. Since $\mathcal{G}$ has property $(q, p)$, $\mathcal{G}|_{U \cup Z}$ contains a clique of size $p$. Hence $\mathcal{G}|_U$ contains a clique of size $p-v$.
\end{proof}

\begin{lemma}\label{lem_hereditary}
Suppose an $r$-uniform hypergraph $\mathcal{F}$ has property $(q, p)$ for all pairs $(q, p)$ with $q \le a\ell$ and $p=\lceil q/a \rceil$ (in other words $\mathcal{F}$ does not have a $(w, v)$-hole with $a\ell\ge w>av$). Then for all $Y \subset \binom{X}{a\ell}$, 
$$\left|\mathcal{F} \cap \binom{Y}{r}\right| \ge a \binom{\ell}{r}.$$
\end{lemma}
\begin{proof}
Instead of this we prove the following stronger statement. Let $(r-1)a \le s \le a\ell$ and $Y \in \binom{X}{s}$.  Suppose further that $s=(a-b)t+b(t-1)$ for some $0\le b<a$, then
$$\left|\mathcal{F} \cap \binom{Y}{r}\right| \ge (a-b)\binom{t}{r}+b\binom{t-1}{r}.$$
Note that the right hand side is $0$ when $s \le (r-1)a$, so the inequality is trivially true in this range. To prove the general case, we use induction on $s$. Since $s =(a-b)t+b(t-1) \in \{at-a+1, \cdots, at\}$,  $\mathcal{F}$ has the $(s, t)$ property from the assumption. Let $R \in \binom{Y}{t}$ span a clique and fix $y \in R$. There are $\binom{t-1}{r-1}$ edges in $\binom{R}{r} \cap \mathcal{F}$ containing $y$. Remove $y$ from $\mathcal{F}$ and apply the inductive hypothesis to $\mathcal{F} \setminus \{y\}$. We infer that
$$\left|\mathcal{F} \cap \binom{Y \setminus \{y\}}{r}\right| \ge (a-b-1)\binom{t}{r} + (b+1)\binom{t-1}{r}.$$ 
Considering the at least $\binom{t-1}{r-1}$ edges containing $y$, we have
\begin{align*}
\left|\mathcal{F} \cap \binom{Y}{r}\right| &\ge (a-b-1)\binom{t}{r} + (b+1)\binom{t-1}{r}+\binom{t-1}{r-1}\\
&=(a-b)\binom{t}{r}+b\binom{t-1}{r}.
\end{align*}
\end{proof}

Now we can proceed as follows to prove Theorem \ref{thm_limit}. The upper bound $\lim_{p \rightarrow \infty} t_r(ap+1, p+1) \le \frac{1}{a^{r-1}}$ is immediate, since $\mathcal{H}_{n, r, a}:=\binom{X_1}{r} \cup \cdots \cup \binom{X_a}{r}$ with $X_1 \sqcup \cdots \sqcup X_a = [n]$, $|X_i|\in \{\lfloor n/a \rfloor, \lceil n/a \rceil\}$ has property $(ap+1, p+1)$ and edge density $1/a^{r-1}+o(1)$. For the remaining of this section we focus on proving the lower bound.

Given $\varepsilon>0$, let us fix a large integer $\ell>\ell_0(a, r, \varepsilon)$, to be determined later. Then fix a much larger integer $L \ge 2a^3\ell^2$, and consider a sufficiently large $r$-uniform hypergraph $\mathcal{F}_0 \subset \binom{[n]}{r}$ having property $(q, p)$ with $q=aL$, $p=L$. Our aim is to find a subset $X \subset [n]$ with $|\binom{X}{r}|>(1-\varepsilon/2) \binom{n}{r}$ such that $\mathcal{F}_0 \cap \binom{X}{r}$ has no $(w, v)$-hole with $w \le a\ell$ and $r-1 \le v$. 

To this end, we start with $\mathcal{F}_0$ and define $\mathcal{F}_i$ inductively. Let $q_0=q, p_0=p, X_0=[n]$. Suppose that $\mathcal{F}_i \subset \binom{X_i}{r}$ has property $(q_i, p_i)$ and it still has a $(w_i, v_i)$-hole. Then we let $Z_i \subset X_i$ be such a $(w_i, v_i)$-hole, and set
$$X_{i+1}=X_i \setminus Z_i, ~~~~~\mathcal{F}_{i+1}=\mathcal{F}_i \cap \binom{X_{i+1}}{r}.$$
By Lemma \ref{lem_hole}, $\mathcal{F}_{i+1}$ has property $(q_i-w_i, p_i-v_i)$. Moreover, the new excess satisfies
$$e(q_i-w_i, p_i-v_i)=a(p_i-v_i)-(q_i-w_i)=(ap_i-q_i)-(av_i-w_i) \ge e(q_i, p_i)+1.$$
Set $q_{i+1}=q_i-w_i$, $p_{i+1}=p_i-v_i$ and continue. At every step
$$a(r-1) \le av_i< |X_i|-|X_{i-1}|=w_i \le a \ell.$$
Suppose at step $i$, the hypergraph $\mathcal{F}_i$ no longer contains a $(w, v)$-hole with $w \le a \ell$. In this case, we choose a subset $Q$ of size $a \ell$ of $V(\mathcal{F}_i)$ uniformly at random. Then by Lemma \ref{lem_hereditary}, 
$$\frac{|\mathcal{F}_i|}{\binom{X_i}{r}}=\frac{\mathbb{E}|\mathcal{F}_i \cap \binom{Q}{r}|}{\binom{a\ell}{r}} \ge \frac{a\binom{\ell}{r}}{\binom{a\ell}{r}}.$$
For sufficiently large $\ell > \ell_0(a, r, \varepsilon)$, this quantity is greater than $(1-\varepsilon/2)\cdot \frac{1}{a^{r-1}}$. On the other hand, $|X_i| \ge n-ia\ell \ge n-pa\ell /(r-1)$. Therefore when $n$ is sufficiently large, $|\binom{X_i}{r}|>(1-\varepsilon/2)\binom{n}{r}$ and therefore
$$|\mathcal{F}_0| \ge |\mathcal{F}_i| \ge (1-\varepsilon/2)\cdot \frac{1}{a^{r-1}} \binom{|X_i|}{r} \ge (1-\varepsilon)\cdot \frac{1}{a^{r-1}}\binom{n}{r}.$$

Otherwise suppose this process continues to produce $(w, v)$-holes. let $m$ be the first index such that $q_m< 2 a \ell$. In view of $e(q_m, p_m) \le (a-1)q_m$ and that $e(q_i, p_i)$ strictly increases after each step, $m \le (a-1)q_m$ follows. Thus
$$aL=q_0 =q_m+\sum_{i=0}^{m-1} w_i \le 2a\ell + ma\ell \le 2a\ell +(a-1)\cdot 2a\ell \cdot a\ell < 2a^3 \ell^2,$$
contradicting $L \ge 2a^3 \ell^2$. 

Summarizing the two cases above, we have that $\lim_{L \rightarrow \infty} t_r(aL, L) \ge 1/a^{r-1}$. Note that a hypergraph having property $(aL+1, L+1)$ must also have property $(aL, L)$. Therefore, 
$$\lim_{p \rightarrow \infty} t_r(ap+1, p+1) \ge 1/a^{r-1}.$$
Together with the construction in the introduction that gives $t_r(ap+1, p+1) \le 1/a^{r-1}$, we conclude the proof of Theorem \ref{thm_limit}.\\

\noindent \textbf{Remark. }Since $\mathcal{H}_{n, r, a}$ also has property $(ap, p)$, we have actually proved a result slightly stronger than Theorem \ref{thm_limit}, namely for every $a, r \ge 2$,
$$\lim_{p \rightarrow \infty} t_r(ap, p) = \frac{1}{a^{r-1}}.$$

\section{The $3$-uniform case}\label{sec_3}

Note that Theorem \ref{thm_limit}, when applied to $a=2$, gives $$\lim_{p \rightarrow \infty} t_r(2p+1, p+1)=\frac{1}{2^{r-1}}.$$ In this section, we determine the exact value of $t_r(2p+1, p+1)$ for $r=3$ and all $p \ge 3$, establishing Theorem \ref{thm_3_uniform}. Our proof is based on two previously known Tur\'an-type results. To apply them, let us change to the complementary notion of excluded configuration.
\begin{definition}
For an $r$-uniform hypergraph $\mathcal{F} \subset \binom{[n]}{r}$. Let $\alpha(\mathcal{F})$ be its independence number, that is, $\alpha(\mathcal{F})=\max\{|A|: A \subset [n], \mathcal{F}\cap \binom{A}{r}=\emptyset\}$.
\end{definition}
Let $\mathcal{F}^c=\binom{[n]}{r} \setminus \mathcal{F}$ be the complementary $r$-uniform hypergraph. Now $\mathcal{F}$ has property $(q, p)$ if and only if $\alpha(\mathcal{H}) \ge p$ for all induced sub-hypergraphs $\mathcal{H}=\mathcal{F}^c \cap \binom{Q}{p}$, $Q \subset [n]$, $|Q|=q$.

For a collection of $\mathcal{G}_1, \cdots, \mathcal{G}_s$ of $r$-uniform hypergraphs, let 
$$t(n, \mathcal{G}_1, \cdots, \mathcal{G}_s)=\max \left\{|\mathcal{F}|: \mathcal{F} \subset \binom{[n]}{r}, ~\mathcal{F}~\textrm{contains no copy of}~\mathcal{G}_i, i=1, \cdots, s \right\}.$$

It is easily seen that $t(n, \mathcal{G}_1, \cdots, \mathcal{G}_s)/\binom{n}{r}$ is a monotone decreasing function of $n$. Consequently  $\lim_{n \rightarrow \infty}t(n, \mathcal{G}_1, \cdots, \mathcal{G}_s)/\binom{n}{r}$ exists. This limit is denoted by $\pi(\mathcal{G}_1, \cdots, \mathcal{G}_s)$, and it is usually called the Tur\'an density of $\{\mathcal{G}_1, \cdots,\mathcal{G}_s\}$.


Consider the following three hypergraphs from \cite{MR}:
$$\mathcal{R}_0=\binom{[4]}{3} \cup \{(a, x, y): a \in [4], x, y \in \{5, 6, 7\}, x \neq y\},$$
$$\mathcal{R}_1=\mathcal{R}_0 \setminus \{\{1,5,6\},\{2,5,7\}, \{3,6,7\}\},$$
$$\mathcal{R}_2=\mathcal{R}_0 \setminus \{\{1,5,6\},\{1,5,7\}, \{3,6,7\}\},$$
It is easy to check that $\alpha(\mathcal{R}_i)=3$ for $i=0,1,2$. To prove $t_3(7,4)=1/4$, it suffices to prove 
\begin{equation}\label{eq_r12}
\pi(\mathcal{R}_1, \mathcal{R}_2)=\frac{3}{4}.
\end{equation}

Actually Mubayi and the third author \cite{MR} proved a considerably stronger statement. Set $\mathcal{R}=\mathcal{R}_0 \setminus \{1,5,6\}$. Then 

\begin{proposition} \label{lem_7}(\cite{MR})
$\pi(\mathcal{R})=\frac{3}{4}$.
\end{proposition}
Since the proof of Proposition \ref{lem_7} is rather short let us include it. Suppose that $\varepsilon>0, n>n_0(\varepsilon)$ and $\mathcal{H} \subset \binom{[n]}{3}$ satisfies $|\mathcal{H}| \ge (3/4+\varepsilon)\binom{n}{3}$. Then for a $4$-element set $Y \subset [n]$ chosen uniformly at random, the expected size of $|\mathcal{H} \cap \binom{Y}{3}|=4|\mathcal{H}|/\binom{n}{3} \ge 3+\varepsilon$. Consequently, $\mathcal{H}$ contains many complete $3$-uniform hypergraphs on $4$ vertices. (As a matter of fact, instead of $3/4$ to ensure that, Razborov \cite{razborov} proved that $0.516\cdots$ would be sufficient to ensure the existence of $K_4^3$.) By symmetry, suppose $\binom{[4]}{3} \subset \mathcal{H}$. For $i \in [4]$ define the link graphs $\mathcal{H}(i)=\{(x, y) \subset [5, n]: (i, x, y) \in \mathcal{H}\}$. Let $\mathcal{G}$ be the multigraph whose edge set is the union (with multiplicities) $\mathcal{H}(1) \cup \cdots \cup \mathcal{H}(4)$. Should $|\mathcal{G}|>3\binom{n-4}{2}+n-6$ hold, we can apply a result of F\"uredi and K\"undgen \cite{furedi-kundgen} which guarantees that there are three vertices in $\mathcal{G}$ spanning at least $11$ edges, which corresponds to a copy of $\mathcal{R}$ in $\mathcal{H}$.
In the opposite case $|\mathcal{H}(i)|<(3/4+\varepsilon/2)\binom{n}{2}$ for some $i \in [4]$, then we remove the vertex $i$ and iterate. Either we find $\mathcal{R}$ or we arrive at a contradiction with $|\mathcal{H}|>(3/4+\varepsilon)\binom{n}{3}$.


The following result was proved by Baber and Talbot \cite{BT} using flag algebra.
\begin{proposition}\label{lem_6}(Theorem $18$ in \cite{BT})
Let $\mathcal{T}$ be the $6$-vertex $3$-uniform vertex hypergraph with $$\mathcal{T}=\binom{[6]}{3} \setminus \{\{1,5,6\}, \{2,4,6\}, \{2,5,6\}, \{3,4,6\}, \{3,4,5\}\}.$$
Then $\pi(\mathcal{T})=3/4.$
\end{proposition}

Now we are ready to prove Theorem \ref{thm_3_uniform}. Observe that if $\mathcal{G}$ and $\mathcal{H}$ are two hypergraphs and $\mathcal{F}$ is their vertex-disjoint union, then $\pi(\mathcal{F})=\max\{\pi(G), \pi(H)\}$.

\begin{proof}[Proof of Theorem \ref{thm_3_uniform}]
We have the upper bound $t_3(2p+1, p+1) \le 1/4$ from \eqref{eq_upper}. Therefore it suffices to establish a matching lower bound. By considering the complement of the host hypergraph, it boils to showing that if the edge density of a $3$-uniform hypergraph $\mathcal{G}$ is greater than $3/4+o(1)$, then $\mathcal{G}$ contains a sub-hypergraph $\mathcal{H}$ on $2p+1$ vertices with $\alpha(\mathcal{H}) \le p$. In other words, we need $\pi(\mathcal{H})\le 3/4$.

For odd $p \ge 3$, we let $\mathcal{H}_1$ be the vertex-disjoint union of $\mathcal{R}$ and $(p-3)/2$ copies of $K_4^3$. It is straightforward to check that $\mathcal{H}_1$ has $7+4 \cdot (p-3)/2=2p+1$ vertices, independence number $3+(p-3)=p$, and $\pi(\mathcal{H}_1) = \max\{\pi(\mathcal{R}), \pi(K_4^3)\}=3/4$. This gives $t_3(2p+1, p+1) \ge 1/4$ for all odd $p \ge 3$.

For even $p \ge 4$,  we take $\mathcal{T}$ from Lemma \ref{lem_6}, and blow up its vertices $1, 2, 3$ twice, and vertices $4, 5, 6$ once to obtain a $9$-vertex hypergraph $\mathcal{T}'$.
Note that a blow-up could only have lower Tur\'an density, therefore  $\pi(\mathcal{T}') \le \pi(\mathcal{T})=3/4$. Moreover the independence number of $\mathcal{T}'$ is $4$, since all the five non-edges of $\mathcal{T}$ contain at most one vertex from $\{1, 2, 3\}$ and $\{4, 5, 6\}$ itself is an edge. We then let $\mathcal{H}_2$ be the vertex-disjoint union of $\mathcal{T}'$ with $(p-4)/2$ copies of $K_4^3$. Then $\mathcal{H}_2$ has $9+4 \cdot (p-4)/2=2p+1$ vertices, $\alpha(\mathcal{H}_2)=4+(p-4)=p$, and $\pi(\mathcal{H}_2) = \max\{\pi(\mathcal{T}'), \pi(K_4^3)\} \le 3/4$. Therefore for all even $p \ge 4$, we also have $t_3(2p+1, p+1) \ge 1/4$. This completes the proof.
\end{proof}

\section{Concluding Remarks} \label{sec_cr}
In this paper, we showed that for $3$-uniform hypergraphs and $p \ge 3$, the $(2p+1, p+1)$ property implies the edge density is at least $1/4-o(1)$. Maybe this can be extended to $r$-uniform hypergraphs and we wonder if the following holds: 

\begin{conjecture}\label{conj_large_p}
For integers $r \ge 2$, and $p$ sufficiently large, 
$$t_r(2p+1, p+1)=\frac{1}{2^{r-1}}.$$
\end{conjecture}
Our Theorem \ref{thm_limit} indicates this is true in the limit, and Theorem \ref{thm_3_uniform} settles the $r=3$ case except for $p=2$, which corresponds to Tur\'an's famous open problem for $K_5^3$. As we
were informed by Sasha Sidorenko \cite{sidorenko}, the $r=4, p=3$ case of Conjecture \ref{conj_large_p} fails to be true since $t_3(7,4) \le 113721/(2^{17}\cdot 10) = 0.08676\cdots  <1/8$.

Here we remark that $\mathcal{T}$ in Lemma \ref{lem_6} with the edge $\{1,4,5\}$ removed still has all the properties needed for the proof of Theorem \ref{thm_3_uniform}. Perhaps one could find a simpler proof that this new hypergraph, much more symmetric than $\mathcal{T}$, still has Tur\'an density $3/4$. Such proof might provide some new insights on the above conjecture.

To determine $t_r(q, p)$, we essentially seek $r$-uniform hypergraph $\mathcal{H}$ with low independence number $\alpha(\mathcal{H})$ relative to its number of vertices, and low Tur\'an density $\pi(\mathcal{H})$. In light of this observation and the results \eqref{eq_graph} and \eqref{eq_small_q}, could it possibly be true that for every positive real number $\gamma>0$,
$$\lim_{p \rightarrow \infty} t_r(\gamma p+1, p+1) = 1-\min_{\mathcal{H} \in \mathcal{F}}\pi(\mathcal{H})=1/\lfloor\gamma\rfloor ^r,$$
where $\mathcal{F}$ is family of all the $r$-uniform hypergraph satisfyings $|V(\mathcal{H}))| \ge \gamma \alpha(\mathcal{H})$?

Finally, motivated by the asymptotic result \eqref{eq_r12} we propose the following conjecture:
\begin{conjecture}
There exists $n_0$ such that for all integers $n>n_0$, 
$$t(2n,\mathcal{R}_1,\mathcal{R}_2)= \binom{2n}{3}-2\binom{n}{3}.$$
\end{conjecture}

\noindent \textbf{Remark. }We would like to thank Alexander Sidorenko for helpful comments on an earlier version of this paper.

\end{document}